\documentclass[12pt]{article}

\usepackage[english]{babel}
\usepackage{amssymb}
\usepackage{graphics}

\newcommand{\qed}{$\;\;\;\Box$}
\newenvironment{proof}{\par\smallbreak{\sl\bf Proof.~}}
{\unskip\nobreak\hfill \qed \par\medbreak}

\newcounter{claim}
\renewcommand{\theclaim}{\arabic{claim}}
{\par\medskip\par}

{\qed\par\smallbreak}
\newcommand{\N}{{\mathbb N}}
\newcommand{\R}{{\mathbb R}}

\newcommand{\beq}{\begin{equation}}
\newcommand{\ee}{\end{equation}}

\renewcommand{\d}{\partial}

\newtheorem{thm}{Theorem}[section]
\newtheorem{lem}[thm]{Lemma}

\newtheorem{defn}[thm]{Definition}

\newtheorem{ex}[thm]{Example}

\newcommand{\be}{\beta}
\newcommand{\ga}{\gamma}

\newcommand{\om}{\omega}

\newcounter{e}
\newcommand{\reff}[1]{(\ref{#1})}

\date{}

\title{Fredholm solvability of time-periodic boundary value hyperbolic problems} 

\newcounter{thesame}
\author{
I. Kmit
\ \ \ R. Klyuchnyk\\
{\small
Institute of Mathematics, Humboldt University of Berlin,}
\\
{\small Rudower Chaussee 25,
\small D-12489 Berlin, Germany and
}
\\
{\small
Institute for Applied Problems of Mechanics and Mathematics,
}
\\
{\small
Ukrainian Academy of Sciences,
Naukova St.\ 3b,
79060 Lviv, Ukraine
}
\\
{\small   E-mail:
{\tt kmit@informatik.hu-berlin.de}}
\\[5mm]
{\small
Institute for Applied Problems of Mechanics and Mathematics,
}
\\
{\small
Ukrainian Academy of Sciences,
Naukova St.\ 3b,
79060 Lviv, Ukraine
}
\\
{\small   E-mail:
{\tt roman.klyuchnyk@gmail.com}}
}

\begin{document}

\maketitle

\begin{abstract}
We investigate a large class of linear boundary value problems for the general first-order one-dimensional hyperbolic systems in the strip $[0,1]\times\R$. We state rather broad natural 
conditions on the data under which the operators of the problems satisfy the Fredholm alternative
 in the spaces of continuous and time-periodic functions. A crucial ingredient of our analysis is a non-resonance condition,
which is  formulated in terms of the data responsible for the bijective part of the 
Fredholm operator.
In the case  of $2\times 2$  systems with reflection boundary conditions,
we provide a criterium for the non-resonant behavior of the system.
\end{abstract}



\section{Introduction}\label{sec:intr} 

\renewcommand{\theequation}{{\thesection}.\arabic{equation}}
\setcounter{equation}{0}

\subsection{Motivation}\label{sec:motiv}

We investigate the general linear first-order hyperbolic system in a single space variable

 \begin{equation}
 \partial_{t}u_j
 +a_j(x,t)\partial_{x}u_j
 +\sum_{k=1}^{n}b_{jk}(x,t)u_k
 =f_j(x,t),
 \;\;\; (x,t)\in(0,1)\times\R,\;\;\; j\le n,
 \label{f1}
 \end{equation}
subjected to periodic conditions in time 
 \begin{equation}\label{f1*}
 u_{j}(x,t)=u_{j}(x,t+2\pi), \;\;\; j\le n, \; t\in\R
 \end{equation}
and boundary conditions in space
  \beq\label{f2}
 \begin{array}{ll}
   u_{j}(0,t)=
 (Ru)_j(t),
 \;\;\; 1\le j\le m,\; t\in \R,  \\ [2mm]
  u_{j}(1,t)=
 (Ru)_j(t),
 \;\;\; m< j\le n,\; t\in \R,
 \end{array}
 \ee
where $0\le m\le n$ are positive integers and $R=(R_1,...,R_n)$ is a bounded linear operator.

From the physical point of view (see Examples \ref{thm:ex3}--\ref{thm:ex5} in Section \ref{sec:ex}), systems of the type \reff{f1}--\reff{f2}
describe models
of laser dynamics \cite{LiRadRe,Peterhof,RadWu,Sieber},  chemical kinetics \cite{aris1,Lyu,zel},
and population dynamics  \cite{eftimie,HRL}.
These systems also have applications in the area of optimal 
boundary control problems \cite{coron,lakra}.

From the mathematical point of view, there is a need for developing  a  theory of
local smooth continuation  \cite{KmRe4} and
 bifurcation \cite{KmRe3} for Fredholm hyperbolic operators, in particular, 
such tools as Lyapunov-Schmidt reduction. Another  source of our  motivation
is developing a stability theory of time-periodic solutions to hyperbolic PDEs,
 in particular, 
such tools as exponential dichotomies. Note that  the known
 theorems about exponential dichotomies 
for ODEs and abstract evolution equations (see, e.g., \cite{LatPS,pal1,pal2})
are stated  in terms of Fredholm solvability. For hyperbolic operators, even
proving a Fredholm property is  a nontrivial issue, and this is the subject that we consider 
in the present paper.

A particular case 
of \reff{f1}--\reff{f2} is studied in \cite{Kirilich}, where an existence result is obtained for 
solutions in the space of continuous and periodic in $t$ functions.
Specifically,  the authors  consider the system \reff{f1}, \reff{f1*} with the boundary conditions
\beq\label{f1100}
\begin{array}{ll}
u_j(0,t)=\mu_j(t), \; 1\le j\le m,  \\ [2mm]
u_j(1,t)=\mu_j(t), \; m<j\le n,
\end{array}
\ee
where $\mu_j(t)$ are time-periodic. An essential assumption made in  \cite{Kirilich} 
is the smallness of all $b_{jk}$. It comes from the 
Banach fixed point argument used in the proof of the main result. 
In the present paper we do not need this assumption and
allow  $b_{jk}$ to be arbitrary elements of the space of continuous and time-periodic functions.
Our main assumption, which is the non-resonance condition \reff{f14} 
stated in Section \ref{sec:setting}, is fulfilled in the setting of \cite{Kirilich}
(this is easy to see after the changing of variables 
$u_j\to v_j=u_j-\mu_j(t)$).

 Time-periodic solutions to the system \reff{f1}
with some reflection boundary conditions are investigated in \cite{KR2,KR3}.
These papers suggest a rather general approach  to proving the Fredholm alternative
in the scale of Sobolev-type spaces of time-periodic functions 
(in the autonomous case  \cite{KR2})
and in the space of continuous and time-periodic functions 
(in the non-autonomous case  \cite{KR3}). In the present paper, we extend the approach from
 \cite{KR3} to a quite general boundary operator $R$ which covers periodic boundary conditions as well as boundary conditions with delays.

\subsection{Our contribution
}\label{sec:setting}

By $C_{n,2\pi}$ 
we  denote the vector space of all $2\pi$-periodic in $t$ and continuous maps $u:[0,1]\times\R\to\R^{n}$, with the norm
 $$
 \|u\|_{\infty}=\max_{j\le n}\max_{x\in [0,1]}\max_{t\in\R}|u_j|.
 $$
 Similarly, $C_{n,2\pi}^1$ denotes the Banach space of all $u\in C_{n,2\pi}$ such that $\d_xu,\d_tu\in C_{n,2\pi}$, with the norm 
 $$
\|u\|_{1}=\|u\|_{\infty}+\|\partial_x u\|_{\infty}+\|\partial_t u\|_{\infty}.
 $$
Also, we use the notation $C_{n,2\pi}(\R)$ for the space of all continuous and $2\pi$-time-periodic maps $v:\R\to\R^n$ and the notation $C_{n,2\pi}^1(\R)$ for the space of all $v\in C_{n,2\pi}(\R)$ with $v^\prime\in C_{n,2\pi}(\R)$. For simplicity, we will skip the subscript $n$ if $n=1$ and write simply $C_{2\pi}$ for $C_{1,2\pi}$ (similarly, we will write $C^1_{2\pi}$, $C_{2\pi}(\R)$, $C^{1}_{2\pi}(\R)$ for $C^{1}_{1,2\pi}$, $C_{1,2\pi}(\R)$, $C_{1,2\pi}^{1}(\R)$, respectively).

We make the following assumptions on the coefficients of  \reff{f1}:
 \beq\label{f4}
 a_j, b_{jk}\in C^1_{2\pi} \mbox{ for all } j\le n \mbox{ and } k\le n,
 \ee
\beq\label{f5}
 a_j(x,t)\neq0  \mbox{ for all } (x,t)\in[0,1]\times\R  \mbox{ and }  j\le n,
\ee
 and 
\beq\label{fz8}
\mbox{for all } 1\le j\neq k\le n \mbox{ there exists } \tilde{b}_{jk}\in C^1_{2\pi}  \mbox{ such that }
b_{jk}=\tilde{b}_{jk}(a_k-a_j).
\ee
The operator   $R$ is supposed to be  a bounded linear operator from $C_{n,2\pi}$ to $C_{n,2\pi}(\R)$
satisfying the following condition: 
\beq\label{ft7}
\begin{array}{ll}
 \mbox{the restriction of the operator } R \mbox{ to } C_{n,2\pi}^1
 \\  
 \mbox{is a bounded linear operator from } C_{n,2\pi}^1 \mbox{ to } C_{n,2\pi}^{1}(\R).
 \end{array}
 \ee

Our goal is to prove the Fredholm alternative for \reff{f1}--\reff{f2}. 
More specifically, we intend to show that, under a certain
 non-resonance condition on 
the coefficients $a_j$, $b_{jj}$, and the boundary operator $R$, either the space of nontrivial solutions to \reff{f1}--\reff{f2} with $f=(f_1,...,f_n)=0$ is not empty and has finite dimension or 
 the system \reff{f1}--\reff{f2} has a unique solution for any $f$.

Let us introduce the characteristics of the hyperbolic system \reff{f1}. Given $j\le n$, $x\in[0,1]$, and $t\in\R$, the $j$-th characteristic is defined as the solution $\xi\in[0,1]\mapsto\omega_j(\xi,x,t)\in\R$ of the initial value problem

\beq\label{f7}
\d_{\xi}\omega_{j}(\xi,x,t)=\frac{1}{a_j(\xi,\omega_{j}(\xi,x,t))}, \;\;\; \omega_{j}(x,x,t)=t.
\ee
To shorten notation, we will simply write $\omega_j(\xi)=\omega_j(\xi,x,t)$. Set 
\beq\label{f8}
c_j(\xi,x,t)=
\exp{{\int_{x}^{\xi}\left (\frac{b_{jj}}{a_j}\right )(\eta,\omega_j(\eta)) d\eta}}, \;\;\;
d_j(\xi,x,t)=
\frac{c_j(\xi,x,t)}{a_j(\xi,\omega_j(\xi))}.
\ee
Integration along the characteristic curves brings the system \reff{f1}--\reff{f2} to the integral form
\begin{eqnarray}
\lefteqn{
 u_j(x,t)=c_j(0,x,t)(Ru)_j(w_j(0))}\nonumber\\ [2mm] &&
 -\int_{0}^{x}d_j(\xi,x,t)\sum_{k\neq j} b_{jk}(\xi,\omega_j(\xi))u_k(\xi,\omega_j(\xi)) d\xi
 +\int_{0}^{x}d_j(\xi,x,t)f_j(\xi,\omega_j(\xi)) d\xi,\nonumber
 \\
&& \hskip10cm 1\le j\le m, \label{f10}
 \\
\lefteqn{
  u_j(x,t)=c_j(1,x,t)(Ru)_j(w_j(1))}\nonumber
 \\ [2mm] &&
 -\int_{1}^{x}d_j(\xi,x,t)\sum_{k\neq j} b_{jk}(\xi,\omega_j(\xi))u_k(\xi,\omega_j(\xi)) d\xi
 +\int_{1}^{x}d_j(\xi,x,t)f_j(\xi,\omega_j(\xi)) d\xi,  \nonumber
 \\
 && \hskip10cm m<j\le n.
 \label{f11}
\end{eqnarray}
By straightforward calculation, one can easily show that a $C^1$-map $u:[0,1]\times\R\to\R^{n}$ is a solution to the PDE problem \reff{f1}--\reff{f2} if and only if it satisfies the system \reff{f10}--\reff{f11}.
This motivates the following definition.

\begin{defn}
A function $u\in C_{n,2\pi}$ is called a {\rm continuous solution} to \reff{f1}--\reff{f2} if it satisfies \reff{f10} and \reff{f11}.
\end{defn}

Introduce an operator $C\in \mathcal{L}(C_{n,2\pi})$ by 
\beq\label{f12}
 (Cv)_j(x,t)=\left\{
 \begin{array}{rl}
c_j(0,x,t)(Rv)_j(\omega_j(0)) &\mbox{for}\ 1\le j\le m,\\
c_j(1,x,t)(Rv)_j(\omega_j(1)) &\mbox{for}\ m<j\le n.
\end{array}
\right.
 \ee

\begin{thm}\label{thm:th12}
Suppose that the conditions \reff{f4}--\reff{ft7} are fulfilled. Assume that there exists $\ell\in\N$ such that
\beq\label{f14}
\|C^{\ell}\|_{\mathcal{L}(C_{n,2\pi})}<1,
\ee
for the operator $C$  defined by \reff{f12}.
Let $\mathcal{K}$ denote the vector space of all continuous solutions to \reff{f1}--\reff{f2} with $f=0$.
Then

$(i)$ $\dim \mathcal{K}<\infty$ and the vector space of all $f\in C_{n,2\pi}$ such that there exists a continuous solution to \reff{f1}--\reff{f2} is a closed subspace of codimension $\dim \mathcal{K}$ in $C_{n,2\pi}$.

$(ii)$ If $\dim \mathcal{K}=0$, then for any $f\in C_{n,2\pi}$ there exists a unique continuous solution $u$ to \reff{f1}--\reff{f2}.
\end{thm}

In Section \ref{sec:remarks} we comment about our crucial conditions \reff{fz8} and \reff{f14} and give examples of the practical cases 
of the problem \reff{f1}, \reff{f2} related to real life applications.
Theorem \ref{thm:th12} 
is proved in Section \ref{sec:Fredh}.  Moreover, in
Section \ref{sec:refl} we consider the case of reflection boundary conditions and
provide  non-resonance conditions that are broader 
than \reff{f14}. 
In the particular case of only two equations in the hyperbolic system \reff{f1},
we derive a necessary and sufficient non-resonance condition, which is stable with 
respect to data perturbations.

\subsection{Further comments}\label{sec:remarks}

\subsubsection{Examples related to applications}\label{sec:ex}

\begin{ex}\label{thm:ex3}\rm
{\it Chemical kinetics.}
The paper \cite{zel} discusses catalytic processes in a chemical reactor. 
A reaction has first order if the reaction rate linearly depends
on the amount of  reactants. In the presence of a catalyst and
 the internal heat exchange, such reactions are described by the following boundary value
problem for a
$3\times 3$-semilinear hyperbolic system:
\beq\label{ex2chem}
\begin{array}{rcl}
\be u_t+u_x&=&KQe^u(1-x)-\ga(u-v),\\
v_t-v_x&=&\ga(u-v),\\
w_t-w_x&=&K(1-x),\\[3mm]
u(0,t)&=&v(0,t),\\
v(1,t)&=&h(t),\\
w(0,t)&=&0,
\end{array}
\ee
where $u$ denotes the temperature in the reactor, 
$v$ is the temperature in the  refrigerator and 
$w$ is the concentration of the reactant.
The positive constants $\ga$, $K$, $\be$, and $Q$ characterize
a catalyst and a reactant.

It is easy to see that  linearizations of \reff{ex2chem} 
are particular cases of \reff{f1}, \reff{f2}.
\end{ex}

\begin{ex}\label{thm:ex4}\rm
{\it Chemotaxis.}
The following correlated random walk model for chemotaxis (chemosensitive movement, see \cite{HRL}) consists of the hyperbolic system 
\beq\label{ex3chem}
\begin{array}{rcl}
\partial_t u^{+}+\partial_x(a_1(x)u^{+})&=&-\mu_1(x)u^{+}+\mu_2(x)u^{-},\\
\partial_t u^{-}-\partial_x(a_2(x)u^{-})&=&-\mu_2(x)u^{-}+\mu_1(x)u^{+}
\end{array}
\ee
and  the boundary conditions 
$$
a^+(x)u^+(x,t) = a^-(x)u^-(x,t),\quad x=0,1,
$$
of the type \reff{f2}.
Here $u^{+}$ and $u^{-}$ are the densities for right and left moving particles.
Furthermore, $\mu_1$, $\mu_2$ are the turning rates and $a_1$, $a_2$ are the particle speeds 
that depend on the external signal $x$.
\end{ex}

\begin{ex}\label{thm:ex5}\rm
{\it Laser dynamics.}
The dynamic behavior of distributed feedback multisection semiconductor lasers is
represented by means of traveling wave models, describing the forward and backward 
propagating complex amplitudes of the light $u=(u_1,u_2)$. The model consists of a 
hyperbolic system coupled to an equation for the carrier density $v$,
namely
$$
\begin{array}{rcl}
\partial_tu(x,t)&=&(-\partial_{x}u_1(x,t),\partial_xu_2(x,t))+G(x,u(x,t),v(x,t)), \nonumber \\
\partial_tv(x,t)&=&\displaystyle I(x,t)+H(x,u(x,t),v(x,t))\nonumber \\
&&\displaystyle 
+\sum_{k=1}^{m}b_{k}\chi_{S_k}(x)\left(\frac{1}{x_k-x_{k-1}}\int_{S_k}v(y,t)dy-v(x,t)\right), \nonumber
\end{array}
$$
which is supplemented with the reflection boundary conditions
$$
\begin{array}{rcl}
u_1(0,t)&=&r_0u_2(0,t)+\alpha(t), \nonumber \\
u_2(1,t)&=&r_{1}u_1(1,t).   \nonumber
\end{array}
$$
Here $0<r_0<1$ and $0<r_1<1$  are reflection coefficients.
This model describes the longitudinal dynamics of edge emitting lasers \cite{LiRadRe}.
A linearization of the  main, hyperbolic part of the model is covered by our system \reff{f1}, \reff{f2}.
\end{ex}

\subsubsection{About the non-resonance condition \reff{f14}}

Suppose that there is $\ell\in\N$ such that $C^{\ell}=0$ in $C_{n,2\pi}$. 
Such boundary conditions appear, for example, in optimal 
 boundary control problems \cite{lakra} and chemical kinetics~\cite{zel}; they are 
\emph{smoothing} in the sense of  \cite{km1,km2,Moyda}.
The condition \reff{f14} is satisfied by trivial reasons in this case,
 and the system \reff{f1}--\reff{f2} is non-resonant.
Even   this case shows that the assumption of Theorem \ref{thm:th12}, involving
the existence of a suitable degree $\ell$, is broader than the condition 
$\|C\|_{\mathcal{L}(C_{n,2\pi})}<1$ (corresponding to $\ell=1$). Indeed,
it is easy to see that, for each $\ell>1$, there is an operator $C$ such that
$C^\ell=0$ while \reff{f14} is not true for any smaller value of~$\ell$.
One can easily check that this is exactly the case for the problem from chemical kinetics
\reff{ex2chem} with $l=2$. Specifically, for the linearization of \reff{ex2chem} at a stationary solution
$(u,v,w)=(u_0(x),v_0(x),w_0(x))$ we have
$$
\begin{array}{rcl}
 (C(u,v,w))_1(x,t)&=&\displaystyle\exp\left\{\be\int_x^0b_{11}(\eta)\,d\eta\right\}v(0,-\be x+t),\\
 (C(u,v,w))_2(x,t)&=&0,\\
 (C(u,v,w))_3(x,t)&=&0,
\end{array}
$$
where $b_{11}(x)=-KQe^{u_0(x)}(1-x)$. Evidently, $C^2=0$.

Consider now practical sufficient conditions making the assumption
\reff{f14} true for small $\ell$. For $\ell=1$ such a condition is
\beq\label{f14RR}
\|R\|_{\mathcal{L}(C_{n,2\pi})}\max_{j,x,t}\exp{\int_{x}^{x_j}\left (\frac{b_{jj}}{a_j}\right )(\eta,\omega_j(\eta)) d\eta}<1.
\ee
This easily follows from  \reff{f12}.

Now consider \reff{f14} for $\ell=2$. Using the notation 
\begin{equation}\label{xj}
 x_j=\left\{
 \begin{array}{rl}
 0 &\mbox{if}\ 1\le j\le m,\\
 1 &\mbox{if}\ m< j\le n
\end{array}
\right.
\end{equation}
and the definition \reff{f12} of the operator $C$, we have
\beq\label{f14CCC}
(C^{2}u)_{j}(x,t)=c_j(x_j,x,t)(RCu)_j(\omega_{j}(x_j)),
\ee
where
\begin{eqnarray*}
(RCu)_j(\omega_{j}(x_j))
=\left(R[c_1(x_1,x,t)(Ru)_1(\omega_{1}(x_1)),...,c_n(x_n,x,t)(Ru)_n(\omega_{n}(x_n))]\right)_j\left (\omega_{j}(x_j)\right ).
\end{eqnarray*}
Therefore, the condition  $\|C^2\|_{\mathcal{L}(C_{n,2\pi})}<1$ follows from
\begin{equation}\label{vghv}
\|RC\|_{\mathcal{L}(C_{n,2\pi})}\max\limits_{j,x,t}\exp{\int_{x}^{x_j}\left(\frac{b_{jj}}{a_j}\right)(\eta,\omega_{j}(\eta))d\eta}<1.
\end{equation}
There are simple examples when \reff{vghv} is true while \reff{f14RR} is not.

\subsubsection{About the conditions \reff{fz8}}
The following two examples show that the condition
 \reff{fz8} plays a crucial role for our result.
\begin{ex}\label{thm:ex1}\rm
Consider the $2\times 2$-system
\begin{equation}\label{one2pi}
 \begin{array}{cc}
\displaystyle\partial_tu_1+\frac{1}{2\pi}\partial_xu_1-u_2=0, & \\ [2mm]
 \displaystyle\partial_tu_2+\frac{1}{2\pi}\partial_xu_2+u_1=0,  &
\end{array}
\end{equation} 
with periodic conditions in both $t$ and $x$, namely
\begin{eqnarray}
& u_{1}(x,t)=u_{1}(x,t+2\pi), \;& u_{2}(x,t)=u_{2}(x,t+2\pi), \label{f3pl}
\\
& u_{1}(x,t)=u_{1}(x+1,t),  \;\;\;& u_{2}(x,t)=u_{2}(x+1,t).\label{f3pl+}
\end{eqnarray}
 This problem is a particular case of \reff{f1}, \reff{f2}
 and satisfies all assumptions of Theorem~\ref{thm:th12} with the exception of \reff{fz8}. 
It is straightforward to check that
$$
 \begin{array}{cc}
u_1=\sin(2\pi x)\sin l(t-2\pi x),\quad l\in\N, \\
u_2=\cos(2\pi x)\sin l(t-2\pi x),\quad l\in\N,
\end{array}
$$
are infinitely many linearly independent solutions to the problem \reff{one2pi}--\reff{f3pl+} and, therefore, the kernel of the operator of \reff{one2pi}--\reff{f3pl+} is infinite dimensional. Thus, the conclusion of Theorem \ref{thm:th12} is not true without \reff{fz8}.
\end{ex}

\begin{ex}\label{thm:ex2}\rm
Consider the $2\times 2$-system
\begin{equation}\label{1ex}
 \begin{array}{cc}
\displaystyle\partial_tu_1+\partial_xu_1=0, \\
\partial_tu_2+\partial_xu_2+bu_1=0,
\end{array}
 \end{equation}
with the periodic conditions in time
$$
 u_{1}(x,t+2\pi)=u_{1}(x,t),\quad u_{2}(x,t+2\pi)=u_{2}(x,t), 
$$
and the reflection conditions in space
$$
 u_{1}(0,t)=r_0u_{2}(0,t),\quad u_{2}(1,t)=r_{1}u_{1}(1,t). 
$$
 Here $r_0$ and $r_1$ are real numbers and $b$ is a non-zero constant. 
If $r_0r_1<1$, then all but \reff{fz8} assumptions of Theorem~\ref{thm:th12} are fulfilled. 
If, moreover,
$$
b=\frac{r_0r_1-1}{r_0},
$$
then
$$
u_1(x,t)=\sin{l(t-x)}, \;\;\; u_2(x,t)=b\left(\frac{1}{1-r_0r_1}-x\right)\sin{l(t-x)}, \; l\in\N,
$$
are infinitely many linearly independent solutions. Again, the conclusion of Theorem \ref{thm:th12} is not true.
\end{ex}

\subsubsection{About the boundary conditions \reff{f2}}

The boundary operator $R$ covers different kinds of reflections, in particular, 
periodic boundary conditions in $x$ and reflection boundary conditions with delays 
(see, e.g., \cite{Lyu} and references therein), for example, if
$$
 (Ru)_{j}(t)=
\sum_{k=1}^{n}\sum_{s=1}^{p}\left[r^0_{jk}(t)u_k(0,t-\theta_{s})+r^1_{jk}(t)u_k(1,t-\theta_{s})
\right], 
\quad j\le n,
$$
where $r_{jk}^0$ and $r_{jk}^1$ are $t$-periodic and continuous functions and $\theta_{s}$ are fixed 
real numbers.

\section{Fredholm alternative (proof of Theorem \ref{thm:th12})}\label{sec:Fredh}

\renewcommand{\theequation}{{\thesection}.\arabic{equation}}
\setcounter{equation}{0}

 Define bounded linear operators $B,F: C_{n,2\pi}\to C_{n,2\pi}$ by
 \beq \label{f34}
 (Bu)_j(x,t)=
 -\int_{x_j}^{x}d_j(\xi,x,t)\sum_{j\neq k} b_{jk}(\xi,\omega_j(\xi))u_k(\xi,\omega_j(\xi)) d\xi, \;\;\; j\le n,
 \ee
 and
 \beq \label{f35}
 (Ff)_j(x,t)=
 \int_{x_j}^{x}d_j(\xi,x,t)f_j(\xi,\omega_j(\xi)) d\xi, \;\;\; j\le n,
 \ee
 where $x_j$ is given by \reff{xj}.
 On the account of \reff{f12}, \reff{f34}, and \reff{f35}, the system \reff{f10}--\reff{f11} can be written as the operator equation
$$
 u=Cu+Bu+Ff.
$$
Note that Theorem \ref{thm:th12} says exactly that the operator $I-C-B: C_{n,2\pi}\to C_{n,2\pi}$ 
is Fredholm of index zero.

 \begin{lem}\label{lem:l1}
The operator $I-C: C_{n,2\pi}\to C_{n,2\pi}$ is bijective.
 \end{lem}
 The proof is a straightforward consequence of the condition \reff{f14} and the Banach fixed-point theorem.

By Lemma \ref{lem:l1}, the operator $I-C-B: C_{n,2\pi}\to C_{n,2\pi}$ 
is Fredholm of index zero if and only if
\beq \label{f37}
I-(I-C)^{-1}B: C_{n,2\pi}\to C_{n,2\pi} \mbox{ is a Fredholm operator of index zero}.
\ee
 To prove \reff{f37}, we will use Nikolsky's criterion of Fredholmness in Banach spaces \cite[Theorem XIII.5.2]{KA}. This criterion  says that an operator $I+K$ on a Banach space is Fredholm of index zero whenever $K^2$ is compact. It is interesting to note that the compactness of $K^2$ 
and the identity $I-K^2=(I+K)(I-K)$ imply 
that the operator $I-K$ is a parametrix of the operator $I+K$; see \cite{zeid}.

We, therefore, have to show that the operator $[(I-C)^{-1}B]^2=(I-C)^{-1}B(I-C)^{-1}B: 
C_{n,2\pi}\to C_{n,2\pi}$ is compact. As the composition of a compact and a bounded operator is a compact operator, it is enough to show that  
$$
B(I-C)^{-1}B: C_{n,2\pi}\to C_{n,2\pi} \mbox{ is  compact.}
$$
Since $B(I-C)^{-1}B=B(I+C+C^2+...)B=B^2+BC(I+C+C^2+...)B=B^2+BC(I-C)^{-1}B$ and $(I-C)^{-1}B$ is bounded, it suffices to prove that
\beq \label{f38}
B^2 \mbox{ and } BC \mbox{ are compact operators from } C_{n,2\pi} \mbox{ to }  C_{n,2\pi}.
\ee

By the Arzela-Ascoli theorem, $C_{n,2\pi}^1$ is compactly embedded into $C_{n,2\pi}$. The desired 
compactness property \reff{f38} will follow if we show that
\beq \label{f309}
B^2 \; \mbox{and} \; BC \mbox{ map continuously } C_{n,2\pi} \mbox{ into }  C_{n,2\pi}^1.
\ee

Using \reff{f12}, \reff{f34} and the equalities
\beq \label{f22}
 \partial_x\omega_{j}(\xi)=
 -\frac{1}{a_j(x,t)}\exp{{\int_{\xi}^{x}\left (\frac{\partial_2a_j}{a_j^2}\right )(\eta,\omega_j(\eta)) d\eta}},
 \ee
 \beq \label{f23}
 \partial_t\omega_{j}(\xi)=
 \exp{{\int_{\xi}^{x}\left (\frac{\partial_2a_j}{a_j^2}\right )(\eta,\omega_j(\eta)) d\eta}},
 \ee
 being true for all $j\le n$, $\xi,x\in[0,1]$, and $t\in\R$, 
we see that
the partial derivatives $\partial_xB^2u$, $\partial_{t}B^2u$, $\partial_xBCu$, $\partial_tBCu$ exist and are continuous for each $u\in C_{n,2\pi}^1$. Here and below by $\d_i$ we denote the partial derivative with respect to the $i$-th argument. Since $C^1_{n,2\pi}$ is dense in $C_{n,2\pi}$, the desired condition \reff{f309} will follow from the next lemma, whose proof will therefore complete proving Theorem \ref{thm:th12}.
\begin{lem}\label{lem:l3}
For all $u\in C^1_{n,2\pi}$ we have
\beq \label{f310}
\left\|B^2u\right\|_{1}
+\|BCu\|_{1}
= O\left(\|u\|_{\infty}\right).
\ee
\end{lem}
\begin{proof}
{\it Claim 1. The following estimate is true:
\beq \label{D2}
\left\|B^2u\right\|_{1}
= O\left(\|u\|_{\infty}\right)\;\textrm{for all} \; u\in C^1_{n,2\pi}.
\ee
}
 Given $j\le n$ and $u\in C^1_{n,2\pi}$, let us consider the following representation for $(B^2u)_j(x,t)$ obtained by application of the Fubini theorem:
\beq \label{f311}
(B^2u)_j(x,t)
=\sum_{k\neq j}\sum_{l\neq k}\int_{x_j}^x\int_{\eta}^{x} d_{jkl}(\xi,\eta,x,t)b_{jk}(\xi,\omega_j(\xi))u_l(\eta,\omega_k(\eta,\xi,\omega_j(\xi))) d\xi d\eta,
\ee
where 
\beq \label{f311d}
d_{jkl}(\xi,\eta,x,t)=d_j(\xi,x,t)d_{k}(\eta,\xi,\omega_{j}(\xi))b_{kl}(\eta,\omega_{k}(\eta,\xi,\omega_j(\xi))).
\ee
From \reff{f311} we immediately get
 the bound
$$
\|B^2u\|_{\infty}=O(\|u\|_{\infty}).
$$
We now claim that 
\beq \label{f31jh1}
\|[(\d_t+a_j(x,t)\d_x)(B^2u)_j]\|_{\infty} = O\left(\|u\|_{\infty}\right)
\mbox{ for all } j\le n \mbox{ and } u\in C^1_{n,2\pi}.
\ee
To prove this, we use  the identity (which follows from \reff{f22} and \reff{f23})
$$
(\d_t+a_j(x,t)\d_x)\varphi(\omega_j(\xi,x,t))\equiv 0,
$$
being true for all $j\le n$, $\varphi\in C^1(\R)$, $x,\xi\in[0,1]$, and $t\in\R$. 
On the account of \reff{f8} and \reff{f311d}, this entails that  for all $j\le n$, $k\le n$, and $l\le n$ we have 
$$
\begin{array}{cc}
(\d_t+a_j(x,t)\d_x)d_{jkl}(\xi,\eta,x,t)\equiv0,
\\ [2mm]
(\d_t+a_j(x,t)\d_x)b_{jk}(\xi,\omega_j(\xi))\equiv0,
\\ [2mm]
(\d_t+a_j(x,t)\d_x)u_l(\eta,\omega_k(\eta,\xi,\omega_j(\xi)))\equiv0.
\end{array}
$$
Using \reff{f311},  we conclude that
\begin{eqnarray*}
\lefteqn{
(\d_t+a_j(x,t)\d_x)(B^2u)_j}\\ &&
=(\d_t+a_j(x,t)\d_x)
\left (\sum_{k\neq j}\sum_{l\neq k}\int_{x_j}^x\int_{\eta}^{x} d_{jkl}(\xi,\eta,x,t)b_{jk}(\xi,\omega_j(\xi))u_l(\eta,\omega_k(\eta,\xi,\omega_j(\xi))) d\xi d\eta\right )
\\ &&
=a_j(x,t)\sum_{k\neq j}\sum_{l\neq k}\int_{x_j}^x d_{jkl}(x,\eta,x,t)b_{jk}(x,\omega_j(x)) u_l(\eta,\omega_k(\eta,x,\omega_j(x))) d\eta
\\ &&
=a_j(x,t)\sum_{k\neq j}b_{jk}(x,t)\sum_{l\neq k}\int_{x_j}^{x} d_{jkl}(x,\eta,x,t) u_l(\eta,\omega_k(\eta)) d\eta.
\end{eqnarray*}
The estimate \reff{f31jh1} now easily follows.

In order to prove \reff{D2}, we have to prove two estimates
\beq\label{x}
\left\|\d_xB^2u\right\|_{\infty} = O\left(\|u\|_{\infty}\right)
\ee
and
\beq \label{f31jhr1}
\|\d_tB^2u||_{\infty} = O\left(\|u\|_{\infty}\right).
\ee
Since \reff{x} follows from \reff{f31jhr1} by \reff{f31jh1} and \reff{f5}, it is enough to prove
\reff{f31jhr1}.

To this end, we start with the following consequence of \reff{f311}:
\begin{eqnarray*}
\lefteqn{
\d_t[(B^2u)_j(x,t)]}
\nonumber\\ &&
=\displaystyle\sum_{k\neq j}\sum_{l\neq k}\int_{x_j}^x\int_{\eta}^{x} \frac{d}{dt}\Bigl[ d_{jkl}(\xi,\eta,x,t)b_{jk}(\xi,\omega_j(\xi))\Bigr] u_l(\eta,\omega_k(\eta,\xi,\omega_j(\xi))) d\xi d\eta
\nonumber\\ &&
+\displaystyle\sum_{k\neq j}\sum_{l\neq k}\int_{x_j}^x\int_{\eta}^{x} d_{jkl}(\xi,\eta,x,t) b_{jk}(\xi,\omega_j(\xi))
\nonumber\\ &&
\times
\d_t\omega_k(\eta,\xi,\omega_j(\xi))\d_t\omega_j(\xi)\d_2u_l(\eta,\omega_k(\eta,\xi,\omega_j(\xi))) d\xi d\eta. \label{f312}
\end{eqnarray*}
Let us transform the second summand. Using \reff{f7}, \reff{f22}, and \reff{f23}, we get
\begin{eqnarray}
\lefteqn{
\frac{d}{d\xi} u_l(\eta,\omega_k(\eta,\xi,\omega_j(\xi)))} \nonumber \\ &&
=\Bigl[\d_x\omega_k(\eta,\xi,\omega_j(\xi))+\d_t\omega_k(\eta,\xi,\omega_j(\xi))\d_{\xi}\omega_j(\xi)\Bigr] \d_2u_l(\eta,\omega_k(\eta,\xi,\omega_j(\xi))) \label{eqwn}
\\ &&
=\left ( \frac{1}{a_j(\xi,\omega_j(\xi))}-\frac{1}{a_k(\xi,\omega_j(\xi))}\right ) \d_t\omega_k(\eta,\xi,\omega_j(\xi))\d_2u_l(\eta,\omega_k(\eta,\xi,\omega_j(\xi))). \nonumber
\end{eqnarray}
Therefore, 
\begin{eqnarray}
\lefteqn{ 
 b_{jk}(\xi,\omega_j(\xi))\d_t\omega_k(\eta,\xi,\omega_j(\xi))\d_2u_l(\eta,\omega_k(\eta,\xi,\omega_j(\xi)))}
\nonumber \\ &&
 =\displaystyle a_j(\xi,\omega_j(\xi))a_k(\xi,\omega_j(\xi))\tilde{b}_{jk}(\xi,\omega_j(\xi))\frac{d}{d\xi} u_l(\eta,\omega_k(\eta,\xi,\omega_j(\xi))), \label{f313}
\end{eqnarray}
where the functions $\tilde{b}_{jk}\in C_{2\pi}$ are fixed so that they satisfy \reff{fz8}. Note that $\tilde{b}_{jk}$ are not uniquely defined by \reff{fz8} for $(x,t)$ with $a_{j}(x,t)=a_{k}(x,t)$. Nevertheless, as it follows from \reff{eqwn}, the right-hand side (and, hence, the left-hand side of \reff{f313}) do not depend on the choice of $\tilde{b}_{jk}$, since $\frac{d}{d\xi}u_{l}(\eta,\omega_{k}(\eta,\xi,\omega_{j}(\xi)))=0$ if $a_{j}(x,t)=a_{k}(x,t)$. 

Write
$$
\tilde{d}_{jkl}(\xi,\eta,x,t)
=d_{jkl}(\xi,\eta,x,t)\d_t\omega_j(\xi)a_k(\xi,\omega_j(\xi))a_j(\xi,\omega_j(\xi))\tilde{b}_{jk}(\xi,\omega_j(\xi)),
$$
where $d_{jkl}$ is introduced by \reff{f311d} and \reff{f8}. Using  \reff{f7} and \reff{f22}, 
we see that the function $\tilde{d}_{jkl}(\xi,\eta,x,t)$ is $C^1$-regular
 in $\xi$ due to the regularity assumptions \reff{f4} and \reff{fz8}. Similarly,
using \reff{f23}, we see that the functions $d_{jkl}(\xi,\eta,x,t)$ and $b_{jk}(\xi,\omega_j(\xi))$ are $C^1$-smooth in $t$.

By \reff{f313} we have
\begin{eqnarray*}
\lefteqn{
\d_t[(B^2u)_j(x,t)]} \\ &&
= \displaystyle \sum_{k\neq j}\sum_{l\neq k}\int_{x_j}^x\int_{\eta}^{x} \frac{d}{dt} [d_{jkl}(\xi,\eta,x,t)b_{jk}(\xi,\omega_j(\xi))] u_l(\eta,\omega_k(\eta,\xi,\omega_j(\xi))) d\xi d\eta
\\ &&
+\displaystyle \sum_{k\neq j}\sum_{l\neq k}\int_{x_j}^x\int_{\eta}^{x}\tilde{d}_{jkl}(\xi,\eta,x,t)\frac{d}{d\xi} u_l(\eta,\omega_k(\eta,\xi,\omega_j(\xi))) d\xi d\eta
\\ &&
=\displaystyle \sum_{k\neq j}\sum_{l\neq k}\int_{x_j}^x\int_{\eta}^{x} \frac{d}{dt} [d_{jkl}(\xi,\eta,x,t)b_{jk}(\xi,\omega_j(\xi))] u_l(\eta,\omega_k(\eta,\xi,\omega_j(\xi))) d\xi d\eta
\\ &&
-\displaystyle \sum_{k\neq j}\sum_{l\neq k}\int_{x_j}^x\int_{\eta}^{x}\d_{\xi}\tilde{d}_{jkl}(\xi,\eta,x,t)u_l(\eta,\omega_k(\eta,\xi,\omega_j(\xi))) d\xi d\eta
\\ &&
+\displaystyle \sum_{k\neq j}\sum_{l\neq k}\int_{x_j}^x\left [\tilde{d}_{jkl}(\xi,\eta,x,t) u_l(\eta,\omega_k(\eta,\xi,\omega_j(\xi)))\right ]_{\xi=\eta}^{\xi=x} d\eta.
\end{eqnarray*}
The desired estimate \reff{f31jhr1} now easily follows from the assumptions \reff{f4}, \reff{f5} and \reff{fz8}.

{\it Claim 2. The following estimate is true:
$$
\|BCu\|_{1}
= O\left(\|u\|_{\infty}\right)\;\textrm{for all} \; u\in C^1_{n,2\pi}.
$$
}
We are done if we  show that 
\beq \label{f314}
\|BCu\|_{\infty}+\|\d_tBCu\|_{\infty} = O\left(\|u\|_{\infty}\right) \;\textrm{for all} \; u\in C^1_{n,2\pi}, 
\ee
as the estimate for $\d_xBCu$ follows similarly to the case of $\d_xB^2u$. 
In order to prove \reff{f314}, we consider an arbitrary  integral contributing into $BCu$,
namely
\beq \label{f315}
\int_{x}^{x_j} e_{jk}(\xi,x,t)b_{jk}(\xi,\omega_j(\xi))(Ru)_k(\omega_k(x_k,\xi,\omega_j(\xi))) d\xi,
\ee
where 
$$
e_{jk}(\xi,x,t)=d_j(\xi,x,t)c_k(x_k,\xi,\omega_j(\xi))
$$
and $j\le n$ and $k\le n$ are arbitrary fixed. From \reff{f315} it follows the bound
$$
\|BCu\|_{\infty}=O(\|u\|_{\infty}).
$$
Differentiating \reff{f315} in $t$, we get
\begin{eqnarray}
\lefteqn{
\displaystyle \int_{x}^{x_j} \frac{d}{dt}\Bigl[e_{jk}(\xi,x,t)b_{jk}(\xi,\omega_j(\xi))\Bigr](Ru)_{k}(\omega_k(x_k,\xi,\omega_j(\xi))) d\xi}
\label{dtDC}
\\ &&
+\displaystyle \int_{x}^{x_j} e_{jk}(\xi,x,t)b_{jk}(\xi,\omega_j(\xi))
\d_t\omega_k(x_k,\xi,\omega_j(\xi))\d_t\omega_j(\xi)(Ru)_k^\prime (\omega_k(x_k,\xi,\omega_j(\xi))) d\xi.\nonumber
\end{eqnarray}
Our task is to estimate the second integral; for the first one the desired estimate is obvious. 
Similarly to the above, we use \reff{f7}, \reff{f22}, and \reff{f23} to obtain 
\begin{eqnarray*}
\lefteqn{
\frac{d}{d\xi} (Ru)_k(\omega_k(x_k,\xi,\omega_j(\xi)))}
\\ &&
=\Bigl[\d_x\omega_k(x_k,\xi,\omega_j(\xi))+\d_t\omega_k(x_k,\xi,\omega_j(\xi))\d_{\xi}\omega_j(\xi)\Bigr](Ru)_{k}^\prime(\omega_k(x_k,\xi,\omega_j(\xi)))
\\ &&
=\left ( \frac{1}{a_j(\xi,\omega_j(\xi))}-\frac{1}{a_k(\xi,\omega_j(\xi))}\right ) \d_t\omega_k(x_k,\xi,\omega_j(\xi))(Ru)_{k}^\prime(\omega_k(x_k,\xi,\omega_j(\xi))).
\end{eqnarray*}
Taking into account \reff{fz8}, the last expression reads
\begin{eqnarray}
\lefteqn{ 
 b_{jk}(\xi,\omega_j(\xi))\d_t\omega_k(x_k,\xi,\omega_j(\xi))(Ru)_{k}^\prime(\omega_k(x_k,\xi,\omega_j(\xi)))}
\nonumber \\ &&
 =\displaystyle a_j(\xi,\omega_j(\xi))a_k(\xi,\omega_j(\xi))\tilde{b}_{jk}(\xi,\omega_j(\xi))\frac{d}{d\xi}(Ru)_{k}(\omega_k(x_k,\xi,\omega_j(\xi))).\label{f313009}
\end{eqnarray}
Set
$$
\tilde{e}_{jk}(\xi,x,t)
=e_{jk}(\xi,x,t)\d_t\omega_j(\xi)a_k(\xi,\omega_j(\xi))a_j(\xi,\omega_j(\xi))\tilde{b}_{jk}(\xi,\omega_j(\xi)).
$$
Using \reff{f313009}, let us transform the second summand in \reff{dtDC} as
\begin{eqnarray}
&\displaystyle \int_{x}^{x_j} e_{jk}(\xi,x,t)b_{jk}(\xi,\omega_j(\xi))
\d_t\omega_k(x_k,\xi,\omega_j(\xi))\d_t\omega_j(\xi)(Ru)^\prime_{k}(\omega_k(x_k,\xi,\omega_j(\xi))) 
d\xi&
\nonumber\\ &
\displaystyle=\int_{x}^{x_j}\tilde{e}_{jk}(\xi,x,t)\frac{d}{d\xi} (Ru)_{k}(\omega_k(x_k,\xi,\omega_j(\xi))) d\xi&
\nonumber\\ &
\displaystyle=\Bigl[\tilde{e}_{jk}(\xi,x,t) (Ru)_k(\omega_k(x_k,\xi,\omega_j(\xi)))\Bigr]_{\xi=x}^{\xi=x_j}&
\nonumber\\ &\displaystyle
-\int_{x}^{x_j}\d_{\xi}\tilde{e}_{jk}(\xi,x,t)(Ru)_k(\omega_k(x_k,\xi,\omega_j(\xi))) d\xi.&
\label{fc22}
\end{eqnarray}
The bound \reff{f314} now easily follows from \reff{dtDC} and \reff{fc22}. 
The lemma is therewith proved.
\end{proof}

\section{Reflection boundary conditions and non-resonant behavior}\label{sec:refl}

\renewcommand{\theequation}{{\thesection}.\arabic{equation}}
\setcounter{equation}{0}

As we have seen in Section~\ref{sec:ex}, in many mathematical models the system \reff{f1} is controlled by the so-called reflection boundary 
conditions. We intend to show that for such problems
the basic assumption \reff{f14}
of Theorem \ref{thm:th12} can be extended.

\subsection{The case of $2\times 2$ systems}

Let \reff{f1} be a system  of two equations, namely
\begin{equation}
\begin{array}{cc}
 \partial_{t}u_1
 +a_1(x,t)\partial_{x}u_1
 +b_{11}(x,t)u_1+b_{12}(x,t)u_2
 =f_1(x,t),\\
 \partial_{t}u_2
 +a_2(x,t)\partial_{x}u_2
 +b_{21}(x,t)u_1+b_{22}(x,t)u_2
 =f_2(x,t),
\end{array}
 \label{f331}
 \end{equation}
 endowed with the periodic conditions in time
\begin{equation}\label{f101}
 u_{j}(x,t)=u_{j}(x,t+2\pi), \;\;\; j=1,2,
 \end{equation}
 and the boundary conditions
\beq \label{FCCC}
\begin{array}{cc}
 u_{1}(0,t)
 =(Ru)_1(t)
 =p_0(t)u_2(0,t),
 \;\;\;
 \\ [2mm]
 u_{2}(1,t)
 =(Ru)_2(t)
 =p_{1}(t)u_1(1,t),
 \;\;\; 
 \end{array}
 \ee
 where $p_{0},p_1\in C_{2\pi}(\R)$. We are able  to derive a sharp non-resonance condition
(ensuring the bijectivity of the operator $I-C$, where $C$ is introduced by \reff{f12}),
which is stable with respect to data perturbations.
Accordingly to \reff{f331}--\reff{FCCC}, the operator 
$C:C_{2,2\pi}\to C_{2,2\pi}$ reads
$$
(Cv)_j(t)=\left\{
\begin{array}{rl}
c_1(0,x,t)p_{0}(\omega_1(0))v_2(0,\omega_1(0)) &\mbox{for}\ j=1,\\
c_2(1,x,t)p_{1}(\omega_2(1))v_1(1,\omega_2(1)) &\mbox{for}\ j=2.
\end{array}
\right.
$$
Then the bijectivity of  $I-C:C_{2,2\pi}\to C_{2,2\pi}$ means that the system
$$
\begin{array}{cc}
u_1(x,t)=c_1(0,x,t)p_0(\omega_1(0))u_2(0,\omega_1(0))
 \\ [3mm]
u_2(x,t)=c_2(1,x,t)p_1(\omega_2(1))u_1(1,\omega_2(1))
\end{array}
$$
has a unique (trivial) solution in $C_{2,2\pi}$ or, the same, the system
$$
\begin{array}{ccl}
u_1(x,t)&=&c_1(0,x,t)p_0(\omega_1(0))c_2(1,0,\omega_1(0))p_1(\omega_2(1,0,\omega_1(0)))u_1(1,\omega_2(1,0,\omega_1(0)))
 \\ [3mm]
u_2(x,t)&=&c_2(1,x,t)p_1(\omega_2(1))u_1(1,\omega_2(1))
\end{array}
$$
has a unique solution in $C_{2,2\pi}$. The first equation at $x=1$ reads
\beq \label{FFAC}
\begin{array}{ccl}
u_1(1,t)&=&c_1(0,1,t)p_0(\omega_1(0,1,t))c_2(1,0,\omega_1(0,1,t))
\\ [3mm]
&&\times p_1(\omega_2(1,0,\omega_1(0,1,t)))u_1(1,\omega_2(1,0,\omega_1(0,1,t))).
\end{array}
\ee
Consider two maps $z(t)=\omega_2(1,0,t)$ and $z(t)=\omega_1(0,1,t)$.
Due to \reff{f5}, both of them  are 
monotonically increasing from $\R$ to $\R$.
Hence, the map $z(t)=\omega_2(1,0,\omega_1(0,1,t))$ is bijective. Moreover, 
the equation \reff{FFAC} is uniquely solvable in $C_{2,2\pi}$ if and only if
$$
|c_1(0,1,t)p_0(\omega_1(0,1,t))c_2(1,0,\omega_1(0,1,t))p_1(\omega_2(1,0,\omega_1(0,1,t)))|\neq 1 \;\textrm{for all}\; t\in\R,
$$
or, the same, if and only if
\beq\label{l333}
\begin{array}{cc}
\displaystyle\exp{\int_{0}^{1}\left[\left (\frac{b_{22}}{a_2}\right )(\eta,\omega_2(\eta,0,\omega_1(0,1,t)))-\left (\frac{b_{11}}{a_1}\right )(\eta,\omega_1(\eta,1,t))\right] d\eta}
\\ [6mm]
\times\left|p_0(\omega_1(0,1,t))p_1(\omega_2(1,0,\omega_1(0,1,t)))\right|\neq 1 \;\textrm{for all}\; t\in\R.
\end{array}
\ee
This is the desired non-resonance condition, which is obviously sharp. Moreover, 
it is stable with respect to data 
perturbation.
Note that, if \reff{l333} is not fulfilled, then 
\reff{f331}--\reff{FCCC} demonstrates the so-called completely resonance behavior.

We also see that the non-resonant behavior of the system \reff{f331}--\reff{FCCC} 
is controlled by the coefficients $a_1,a_2$  of the differential part and by the 
coefficients $b_{11},b_{22}$ of the diagonal lower order part of the hyperbolic system, 
as well as by the reflection coefficients $p_0,p_1$.

\subsection{The case of $n\times n$ systems}

Let us consider the system \reff{f1} with the reflection boundary conditions
\beq \label{f42}
\begin{array}{ccl}
 u_{j}(0,t)&=&
 \displaystyle\sum_{k=m+1}^{n}p_{jk}(t)u_k(0,t)+\sum_{k=1}^{m}p_{jk}(t)u_k(1,t),
 \;\;\; 1\le j\le m,
 \\ [4mm]
 u_{j}(1,t)&=&
 \displaystyle\sum_{k=m+1}^{n}p_{jk}(t)u_k(0,t)+\sum_{k=1}^{m}p_{jk}(t)u_k(1,t),
 \;\;\; m<j\le n,
 \end{array}
 \ee
 where $p_{jk}\in C_{2\pi}(\R)$. Then the operator $C: C_{n,2\pi}\to C_{n,2\pi}$ reads
 $$
\begin{array}{cc}
 (Cv)_j(x,t)
 = c_j(x_j,x,t)\left[\displaystyle(1-x_j)\sum_{k=1}^{n}p_{jk}(\omega_j(0))v_k(1-x_k,\omega_j(0))\right.
 \\ [3mm]
 +\left.\displaystyle x_j\sum_{k=1}^{n}p_{jk}(\omega_j(1))v_k(1-x_k,\omega_j(1))\right], \;\;\; j\le n.
 \end{array}
 $$
Introduce the functions
$$
S_j(t)=\left\{
\begin{array}{rl}
 \displaystyle
c_j(0,1,t)\sum\limits_{k=1}^{n}|p_{jk}(\omega_j(0,1,t))| &\mbox{for}\ 1\le j\le m,\\
 \displaystyle c_j(1,0,t)\sum\limits_{k=1}^{n}|p_{jk}(\omega_j(1,0,t))| &\mbox{for}\ m<j\le n.
\end{array}
\right.
$$
A non-resonance condition analogous to \reff{f14} can be 
stated as
\beq \label{f12sff99}
\max_{j\le n}\max_{t\in \R} S_j(t)<1.
 \ee
Using the strategy of the proof of Theorem \ref{thm:th12}, let us show that under the condition \reff{f12sff99} the system
\beq\label{sds}
\begin{array}{cc}
u_j(x,t)=
\displaystyle c_j(x_j,x,t)\left[(1-x_j)\sum_{k=1}^{n}p_{jk}(\omega_j(0))u_k(1-x_k,\omega_j(0))\right.
 \\ 
 +\displaystyle \left.x_j\sum_{k=1}^{n}p_{jk}(\omega_j(1))u_k(1-x_k,\omega_j(1))\right], \;\;\; j\le n
\end{array}
\ee
is uniquely solvable in $C_{n,2\pi}$ with respect to $u_{j}, j\le n$. Putting $x=0$ for $m< j\le n$ and $x=1$ for $j\le m$ in \reff{sds}, we get the following system of $n$ equations with respect to $n$ unknowns $u_j(0,t)$, $m<j\le n$ and $u_j(1,t)$, $j\le m$:

\beq \label{lsl}
\begin{array}{cc}
u_j(0,t)=
\displaystyle c_j(1,0,t)\sum_{k=1}^{n}p_{jk}(\omega_j(1,0,t))u_k(1-x_k,\omega_j(1,0,t)), \;\;\; m<j\le n,
 \\ [3mm]
u_j(1,t)=
\displaystyle c_j(0,1,t)\sum_{k=1}^{n}p_{jk}(\omega_j(0,1,t))u_k(1-x_k,\omega_j(0,1,t)), \;\;\; 1\le j\le m.
\end{array}
\ee
Notice that the unique solvability of \reff{lsl} in $C_{n,2\pi}(\R)$ entails the  unique solvability of
\reff{sds} in $C_{n,2\pi}$. From \reff{lsl} we have
\beq\label{dfdfdf}
\begin{array}{cc}
\max\limits_{m<i\le n}\max\limits_{j\le m}\max\limits_{t,\tau\in\R}\left\{|u_i(0,t)|, |u_j(1,\tau)|\right\}
\\ [3mm]
\le\displaystyle \max_{t,\tau\in\R}\left\{
\max_{j\le m}c_j(0,1,\tau)\sum_{k=1}^{n}|p_{jk}(\omega_j(0,1,\tau))||u_k(1-x_k,\omega_j(0,1,\tau))|\right., 
\\ [3mm]
\left.\displaystyle \max_{m<j\le n}c_j(1,0,t)\sum_{k=1}^{n}|p_{jk}(\omega_j(1,0,t))||u_k(1-x_k,\omega_j(1,0,t))|\right\}
\\ [3mm]
\le\max\limits_{j\le n}\max\limits_{t\in\R}S_j(t)\max\limits_{k\le m}\max\limits_{m<i\le n}\max\limits_{t,\tau\in\R}\{|u_i(0,t)|, |u_k(1,\tau)|\}.
\end{array}
\ee
Using \reff{dfdfdf} and applying the Banach fixed-point argument to \reff{lsl}, we conclude that \reff{f12sff99} ensures the unique solvability of \reff{sds}, as desired. 

We now show, in addition to \reff{f12sff99}, another sufficient non-resonance condition.
To this end, we change the variable 
$t$ to $\tau=\om_j(1,0,t)$ for $m<j\le n$ and
$t$ to $\tau=\om_j(0,1,t)$
for $j\le m$ in \reff{lsl}. 
This allows us to rewrite the system \reff{lsl} as follows:
\beq \label{lsl1}
\begin{array}{rcll} 
\displaystyle u_j(0,\omega_j(1,0,\tau))&=&\displaystyle
\displaystyle c_j(1,0,\omega_j(1,0,\tau))
 \sum_{k=1}^{n}p_{jk}(\tau)u_k(1-x_k,\tau), \quad& m<j\le n,
 \\ [3mm]
\displaystyle u_j(1,\omega_j(0,1,\tau))&=&\displaystyle
 c_j(0,1,\omega_j(0,1,\tau))\sum_{k=1}^{n}p_{jk}(\tau)u_k(1-x_k,\tau),\quad& j\le m.
\end{array}
\ee
Set $v(t)=(u_1(1,t),...,u_m(1,t),u_{m+1}(0,t),...,u_{n}(0,t))$ and rewrite \reff{lsl1} in the operator-matrix form
$$
(Gv)(t)=Q(t)v(t),
$$
where the operator $G\in\mathcal{L}(C_{n,2\pi}(\R))$ is given by
$$
(Gv)(t)=(u_1(1,\omega_j(0,1,\tau)),...,u_m(1,\omega_j(0,1,\tau)),u_{m+1}(0,\omega_j(1,0,\tau)),...,u_{n}(0,\omega_j(1,0,\tau)))
$$
and the matrix $Q(t)$ is defined by the right-hand side of \reff{lsl1}. Assume that the matrix $Q(t)$ is invertible for all $t\in\R$, and, moreover,
\begin{equation}\label{Q12}
\|Q^{-1}(t)\|_{\infty}<1.
\end{equation}
Then the system \reff{lsl1} and, hence, the system \reff{sds} is uniquely solvable. This means that 
\reff{Q12} is, additionally to \reff{f12sff99},   a non-resonance condition for 
the problem \reff{f1}, \reff{f42}. 

To illustrate applicability of these two non-resonance conditions, suppose that 
the  coefficients $a_j$, $b_{jj}$, and $p_{jk}$ are constant. In this case the condition \reff{f12sff99} 
is simplified to
$$
\exp{\left\{\frac{(-1)^{1-x_j}b_{jj}}{a_j}\right\}}\sum_{k=1}^{n}|p_{jk}|<1\quad \mbox{ for all }
j\le n.
$$
The matrix $Q$ in this case does not depend on  $t$ and reads
$$
Q=\left(
\begin{array}{ccc}
p_{11}\exp\left\{- \frac{b_{11}}{a_1}\right\}& \ldots & p_{1n}\exp{\left\{-\frac{b_{11}}{a_1}\right\}}\\
\vdots & \ddots & \vdots\\
p_{m1}\exp{\left\{-\frac{b_{mm}}{a_m}\right\}} & \ldots & p_{mn}\exp{\left\{-\frac{b_{mm}}{a_m}\right\}}\\
p_{m+1,1}\exp{\left\{\frac{b_{m+1,m+1}}{a_{m+1}}\right\}} & \ldots & p_{m+1,n}\exp{\left\{\frac{b_{m+1,m+1}}{a_{m+1}}\right\}}
\\
\vdots & \ddots & \vdots\\
p_{n1}\exp{\left\{\frac{b_{nn}}{a_n}\right\}} & \ldots & p_{nn}\exp{\left\{\frac{b_{nn}}{a_n}\right\}}
\end{array}
\right).
$$
If $Q$ is invertible and the norm of $Q^{-1}$ is less than one, then we meet our second non-resonance condition
 \reff{Q12}.

\section*{Acknowledgments}
The second author was supported by the BMU-MID Erasmus Mundus Action 2 grant. He expresses his gratitude to the Applied Analysis group at the Humboldt University of Berlin for its kind hospitality.

\end{document}